\documentclass[11pt]{article}

\usepackage[T1]{fontenc} 
\usepackage{amsmath,amsfonts,amssymb,amsthm,mathtools}
\usepackage{float}
\usepackage{fullpage}
\usepackage{hyperref,cleveref}
\usepackage{tikz,graphicx,xcolor,caption,subcaption} 

\allowdisplaybreaks 

\usetikzlibrary{calc}

\definecolor{matchcolor}{HTML}{20252B}
\definecolor{kzerocolor}{HTML}{1479B8}
\definecolor{konecolor}{HTML}{D97706}

\newtheorem{theorem}{Theorem}[section]
\newtheorem{lemma}[theorem]{Lemma}
\newtheorem{corollary}[theorem]{Corollary}
\newtheorem{proposition}[theorem]{Proposition}
\theoremstyle{definition}

\theoremstyle{remark}

\newcommand{\BlockPairDiagram}{%
\begin{tikzpicture}[x=1.5cm,y=0.60cm,line cap=round,line join=round]
  \foreach \side/\x/\name in {L/0/i,R/1/{i+1}}{
    \coordinate (\side-00) at (\x,0);
    \coordinate (\side-01) at (\x,1);
    \coordinate (\side-10) at (\x,2);
    \coordinate (\side-11) at (\x,3);
    \node[font=\small] at (\x,3.72) {$B_{\name}$};
  }
  \foreach \u in {00,10}{\foreach \v in {00,01}{
    \draw[kzerocolor,line width=0.72pt] (L-\u)--(R-\v);
  }}
  \foreach \u in {01,11}{\foreach \v in {10,11}{
    \draw[konecolor,line width=0.72pt] (L-\u)--(R-\v);
  }}
  \foreach \side in {L,R}{
    \draw[matchcolor,line width=1.05pt] (\side-00)--(\side-01);
    \draw[matchcolor,line width=1.05pt] (\side-10)--(\side-11);
  }
  \foreach \lab in {00,01,10,11}{
    \fill[black] (L-\lab) circle[radius=0.047];
    \fill[black] (R-\lab) circle[radius=0.047];
    \node[font=\scriptsize,anchor=east,xshift=-3pt] at (L-\lab) {$\lab$};
    \node[font=\scriptsize,anchor=west,xshift=3pt] at (R-\lab) {$\lab$};
  }
\end{tikzpicture}%
}

\newcommand{\GraphRing}[1]{%
\begin{tikzpicture}[line cap=round,line join=round]
  \pgfmathtruncatemacro{\NB}{2*#1}
  \pgfmathtruncatemacro{\LAST}{\NB-1}
  \pgfmathsetmacro{\RR}{0.42+0.105*#1}
  \foreach \i in {0,...,\LAST}{
    \pgfmathsetmacro{\ang}{90-360*\i/\NB}
    \coordinate (C-\i) at (\ang:\RR);
    \begin{scope}[shift={(C-\i)},rotate={\ang-90}]
      \coordinate (v-\i-0-0) at (0,-0.27);
      \coordinate (v-\i-0-1) at (0,-0.09);
      \coordinate (v-\i-1-0) at (0, 0.09);
      \coordinate (v-\i-1-1) at (0, 0.27);
    \end{scope}
    \coordinate (L-\i) at (\ang:{\RR+0.58});
  }
  \foreach \i in {0,...,\LAST}{
    \pgfmathtruncatemacro{\j}{mod(\i+1,\NB)}
    \foreach \a in {0,1}{\foreach \c in {0,1}{
      \draw[kzerocolor!78,line width=0.34pt]
        (v-\i-\a-0)--(v-\j-0-\c);
      \draw[konecolor!82,line width=0.34pt]
        (v-\i-\a-1)--(v-\j-1-\c);
    }}
  }
  \foreach \i in {0,...,\LAST}{
    \foreach \a in {0,1}{
      \draw[matchcolor,line width=0.82pt]
        (v-\i-\a-0)--(v-\i-\a-1);
    }
    \foreach \a in {0,1}{\foreach \b in {0,1}{
      \fill[black] (v-\i-\a-\b) circle[radius=0.038];
    }}
    \node[font=\scriptsize] at (L-\i) {$B_{\i}$};
  }
\end{tikzpicture}%
}

\title{Induced-saturated graphs exist for even cycles}
\author{
{{Ilkyoo Choi}}\thanks{
\footnotesize {Department of Mathematics, Hankuk University of Foreign Studies, Yongin-si, Gyeonggi-do, Republic of Korea.
 E-mail: \texttt {ilkyoo@hufs.ac.kr},
  Discrete Mathematics Group, Institute for Basic Science (IBS), Daejeon, Republic of Korea, and Center for AI and Natural Sciences, Korea Institute for Advanced Study (KIAS), Seoul, Republic of Korea.
Supported by the Hankuk University of Foreign Studies Research Fund, the National Research Foundation of Korea (NRF) grant funded by the Korea government (MSIT) (RS-2025-23324220), Institute for Basic Science (IBS-R029-C1), and the Korea Institute for Advanced Study (KIAS) grant funded by the Korea government.
}}}
\date{}

\begin{document}
\maketitle

\begin{abstract}
A graph $G$ is \emph{$H$-induced-saturated} if $G$ has no induced subgraph isomorphic to $H$ but changing the adjacency of an arbitrary pair of vertices in $G$ creates an induced copy of $H$. 
The existence problem for $H$-induced-saturated graphs had previously been settled when $H$ is a complete graph, a path, an odd cycle, or an even cycle of length at most $10$. 
In this paper, for every integer $q\ge3$, we construct a $C_{2q+2}$-induced-saturated graph.
Hence, induced-saturated graphs exist for all cycles, except for the cycle of length 3.
\end{abstract}

\section{Introduction}

All graphs in this paper are finite and simple, which means no loops and no parallel edges. 
Let $K_n, P_n, C_n$ denote the complete graph, path, cycle, respectively, on $n$ vertices. 
Given a graph $G$, let $V(G)$ and $E(G)$ denote its vertex set and edge set, respectively. 

Given a graph $H$, a graph $G$ is \emph{$H$-saturated} if $G$ has no subgraph isomorphic to $H$, but adding an edge between a pair of non-adjacent vertices in $G$ creates a subgraph isomorphic to $H$. 
Note that removing an edge of $G$ cannot create a copy of $H$ if $G$ has no copy of $H$. 
Adding edges between pairs of non-adjacent vertices of a graph on at least $|V(H)|$ vertices with no copy of $H$ will eventually create a copy of $H$, so there exists an $H$-saturated graph for every $H$. 
Naturally, the next question is to ask for the maximum or minimum edge density of $H$-saturated graphs.
We direct the readers to a monograph by Bollob\'as~\cite{2004Bollobas} for the literature on this topic; see also the surveys~\cite{2011CuFaFaSc,2013FuSi}. 

The notion of $H$-induced-saturated graphs had been studied earlier, but the term was coined by Axenovich and Csik\'{o}s~\cite{2019AxCs}. 
Given a graph $H$, a graph $G$ is \emph{$H$-induced-saturated} if $G$ has no induced subgraph isomorphic to $H$, but changing the adjacency of an arbitrary pair of vertices in $G$ creates an induced subgraph isomorphic to $H$. 
In other words, $G$ has no induced copy of $H$, every addition of an edge between a pair of non-adjacent vertices of $G$ creates an induced copy of $H$ and every deletion of an edge of $G$ creates an induced copy of $H$.
In contrast to $H$-saturated graphs, which always exist on sufficiently many vertices, $H$-induced-saturated graphs need not exist. 
For example, no $K_n$-induced-saturated graph exists for $n\geq 3$ since removing an edge cannot create a new copy of $K_n$. 

The existence of induced-saturated graphs for paths is completely settled. 
The empty graph on at least two vertices is $P_2$-induced-saturated, and the disjoint union of complete graphs where each component has at least three vertices is $P_3$-induced-saturated. 
Perhaps surprisingly, Martin and Smith~\cite{2012MaSm} observed that $P_4$-induced-saturated graphs do not exist.
Shortly after Axenovich and Csik\'os~\cite{2019AxCs} explicitly asked the question of existence of induced-saturated graphs for paths, R\"{a}ty~\cite{2020Raty} proved that a $P_6$-induced-saturated graph exists. 
Cho, Choi, and Park~\cite{2021ChChPa} showed that induced-saturated graphs exist for infinitely many paths; namely, they constructed an infinite sequence of graphs $G_n$ where each $G_n$ is $P_{3n}$-induced-saturated for a positive integer $n$. 
They also showed that the Kneser graph $K(n, 2)$ is $P_6$-induced-saturated for every $n\geq 5$.
Inspired by this construction, Dvo\v{r}\'{a}k~\cite{2020Dvorak} proved that induced-saturated graphs exist for all paths on at least six vertices. 
The last case of $P_5$ was settled by Bonamy et al.~\cite{unpub_BoGrJoMoSc} via a computer search, finding five $P_5$-induced-saturated graphs.

The next natural family of graphs to investigate is the family of cycles. 
Recall that there is no $C_3$-induced-saturated graph since $C_3$ is isomorphic to $K_3$. 
Behrens et al.~\cite{2016BeErSaYaYe} showed that the line graph of the complete bipartite graph $K_{r,r}$ is $C_{2r-1}$-induced-saturated when $r\geq 3$. 
They also constructed a $C_4$-induced-saturated graph. 
Recently, constructions for $C_{2r}$-induced-saturated graphs were discovered for $r\in\{3,4,5\}$, see~\cite{arXiv_2025FaHaHaSp}.
It remained unknown whether induced-saturated graphs exist for the other even cycles.
In this paper, we completely resolve this question by constructing a $C_{2q+2}$-induced-saturated graph for every $q\geq 3$. 
Together with the aforementioned results, this completes the existence problem for cycles: $C_3$ is the only cycle for which no induced-saturated graph exists.

\begin{theorem}\label{thm:main}
For every integer $q\ge3$, there is a $C_{2q+2}$-induced-saturated graph $G_q$.
\end{theorem}

We prove \Cref{thm:main} in the next four sections. 
In \Cref{sec:properties}, we define $G_q$ and reveal some structural properties of $G_q$.
We show that $G_q$ has no induced copy of $C_{2q+2}$ in \Cref{sec:nocycle}.
In \Cref{sec:addedge} and \Cref{sec:deleteedge}, we show that adding an edge between a pair of non-adjacent vertices and deleting an edge, respectively, of $G_q$ always creates an induced copy of $C_{2q+2}$.

\section{Properties of the graph $G_q$}
\label{sec:properties}

We first explain the construction. 
Fix $q\ge3$ and let $n=8q$.
We now define the graph $G_q$ on $n$ vertices: 
let $V(G_q)=\{i^{ab}:i\in\mathbb Z_{2q},\ a,b\in\{0,1\}\}$, and define the $i$th \emph{block} $B_i$ to be
$\{i^{00},i^{01},i^{10},i^{11}\}$. 
We often refer to $a,b$ as \emph{bits}. 
For convenience, we abuse notation and write $i^{ab}$ to mean $(i\pmod{2q})^{ab}$.
Similarly, all block indices are modulo $2q$.  
The graph $G_q$ has two types of edges, which we call matching edges and cross edges. 

\begin{enumerate}
    \item\label{eq:M}
    Matching edges inside a block $B_i$: 
  $i^{ab}i^{a,1-b}$.
  
Thus, in each block $B_i$, there are two matching edges $i^{00}i^{01}$ and $i^{10}i^{11}$.

\item\label{eq:S} Cross edges between consecutive blocks $B_i$ and $B_{i+1}$: $i^{ab}(i+1)^{bc}$.
  
Thus, $i^{ab}(i+1)^{cd}\in E(G_q)$ if and only if $b=c$.
\end{enumerate}

The $i$th \emph{connection} is the set of cross edges between $B_i$ and $B_{i+1}$.  
Every connection forms the disjoint union of two copies of $K_{2,2}$.
For fixed $a$ and $i$, the two vertices  $i^{a0}, i^{a1}$ (in $B_i$) have the same two neighbors in $B_{i-1}$, and their neighborhoods in $B_{i+1}$ are disjoint.

Every vertex of $G_q$ has one matching neighbor, two neighbors in the preceding block, and two neighbors in the following block.
Therefore $G_q$ is $5$-regular.
See \Cref{fig:firstgraphs}.
\begin{figure}[H]
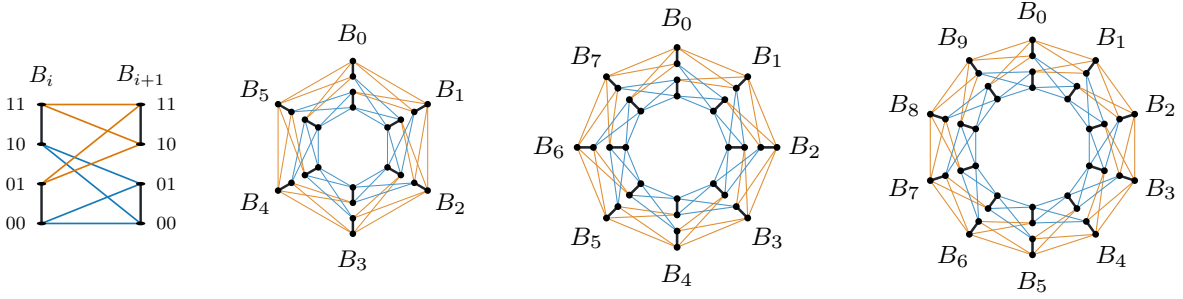

\centering
\begin{tabular}{@{}c@{\hspace{1.5em}}c@{\hspace{1.5em}}c@{\hspace{1.5em}}c@{}}
\raisebox{-0.5\height}{%
  \resizebox{0.15\textwidth}{!}{\BlockPairDiagram}} &
\raisebox{-0.5\height}{%
  \resizebox{0.20\textwidth}{!}{\GraphRing{3}}} &
\raisebox{-0.5\height}{%
  \resizebox{0.25\textwidth}{!}{\GraphRing{4}}} &
\raisebox{-0.5\height}{%
  \resizebox{0.25\textwidth}{!}{\GraphRing{5}}}
\end{tabular}

\caption{The edges in consecutive blocks, and the graphs $G_3$, $G_4$, and $G_5$.}
\label{fig:firstgraphs}
\end{figure}

Let $\varepsilon=(\varepsilon_0,\varepsilon_1,\ldots,\varepsilon_{2q-1})\in\{0,1\}^{2q}$.
For bits $r,s\in\{0,1\}$,  $r\oplus s$ means $r+s\pmod 2$.
In particular, $r\oplus0=r$ and 
  $r\oplus1=1-r$.
Define
\[
  \rho_t(i^{ab})=(i+t)^{ab}
  \text{ and }
  \rho_\varepsilon(i^{ab})
    =i^{\,a\oplus\varepsilon_i,\,
    b\oplus\varepsilon_{i+1}}.
\]
All subscripts of $\varepsilon$ are read modulo $2q$.

In words, $\rho_t$ rotates the blocks by $t$ positions.
The map $\rho_\varepsilon$ is best viewed as an independent relabeling of the vertices in consecutive blocks.
Each $\varepsilon_i$ specifies whether the labels $0$ and $1$ are kept fixed or interchanged for vertices in $B_{i-1}$ and $B_i$.
A switch at index $j$ interchanges the two matching edges in $B_j$ and flips each matching edge in $B_{j-1}$.
Let us elaborate. 
Suppose that $\varepsilon_j=1$ and all other entries of $\varepsilon$ are $0$.  
In the block $B_j$, the first bit is changed, so $j^{0b}\leftrightarrow j^{1b}$.
Thus the two matching edges $j^{00}j^{01}$ and $j^{10}j^{11}$ are switched with each other.  
In the preceding block $B_{j-1}$, the second bit is changed, so $(j-1)^{a0}\leftrightarrow (j-1)^{a1}$.
Thus the two endpoints of each matching edge in $B_{j-1}$ are interchanged.
No other blocks are affected.

For a bijection $\phi:V(G_q)\to V(G_q)$, write $\phi(G_q)$ for the graph obtained by replacing every vertex $v$ by $\phi(v)$.
Thus, $E\bigl(\phi(G_q)\bigr)
  =\{\phi(u)\phi(v):uv\in E(G_q)\}$.

\begin{lemma}
\label{lem:relabelings}
For every $t\in\mathbb Z_{2q}$ and every $\varepsilon\in\{0,1\}^{2q}$,
$\rho_t(G_q)= G_q$ and $\rho_\varepsilon(G_q)= G_q$.
In other words, both $\rho_t$ and $\rho_\varepsilon$ preserve adjacency and non-adjacency.
\end{lemma}

\begin{proof}
First consider $\rho_t$. 
It is a bijection with inverse $\rho_{-t}$.
A matching edge $i^{ab}i^{a,1-b}$ is sent to $(i+t)^{ab}(i+t)^{a,1-b}$, which is again a matching edge.  
A cross edge $i^{ab}(i+1)^{bc}$ is sent to $(i+t)^{ab}(i+t+1)^{bc}$, which is again a cross edge.  
Thus every edge of $\rho_t(G_q)$ is an edge of $G_q$.  
Applying the same argument to the inverse $\rho_{-t}$ gives the reverse inclusion, so $\rho_t(G_q)=G_q$.

Now consider $\rho_\varepsilon$. 
Since $(r\oplus\varepsilon_j)\oplus\varepsilon_j=r$ for every pair of bits $r,\varepsilon_j$, applying
$\rho_\varepsilon$ twice is the identity. 
Thus $\rho_\varepsilon$ is a bijection and is its own inverse.
A matching edge $i^{ab}i^{a,1-b}$ is sent to $i^{a\oplus\varepsilon_i, b\oplus\varepsilon_{i+1}}i^{a\oplus\varepsilon_i,(1-b)\oplus\varepsilon_{i+1}}$.
Since $(1-b)\oplus\varepsilon_{i+1}=1-(b\oplus\varepsilon_{i+1})$, it is a matching edge.
A cross edge $i^{ab}(i+1)^{bc}$ is sent to $i^{\,a\oplus\varepsilon_i,\,b\oplus\varepsilon_{i+1}}(i+1)^{\,b\oplus\varepsilon_{i+1},\,c\oplus\varepsilon_{i+2}}$, so it is a cross edge.
Thus every edge of $\rho_\varepsilon(G_q)$ is an edge of $G_q$.  
Since $\rho_\varepsilon$ is its own inverse, $\rho_\varepsilon(G_q)=G_q$.
\end{proof}

\begin{corollary}
\label{cor:normalization}
There are two types of edges and $q+2$ types of non-edges. 
In particular, 
\begin{enumerate}
\item Every matching edge is of type $0^{00}0^{01}$ and every cross edge is of type $0^{00}1^{00}$.
\item Every non-edge is of type $0^{00}0^{10}, 0^{00}0^{11},  0^{00}1^{10}$, or $0^{00}d^{00}$ for $d\in\{2, \ldots, q\}$.
\end{enumerate}
\end{corollary}

\begin{proof}
Let $u=i^{ab}$ and $v=j^{cd}$ be distinct vertices, and let  $d_0=
\min\{|k|:k\in\mathbb Z,\ i+k\equiv j\pmod{2q}\}$.
In other words, $d_0$ is the cyclic block distance. 
By interchanging $u$ and $v$, if necessary, we may assume $j\equiv i+d_0\pmod{2q}$.
Applying the rotation $\rho_{-i}$, we may further assume that
$i=0$ and $j=d_0$.
We now choose the entries of $\varepsilon$.

\textbf{Case 1}: $d_0=0$, so $u$ and $v$ lie in the same block.

Choose $(\varepsilon_0, \varepsilon_1)=(a,b)$.
Then $\rho_\varepsilon(u)=\rho_\varepsilon(0^{ab})=0^{a\oplus a, b\oplus b}=0^{00}$ and $\rho_\varepsilon(v)=\rho_\varepsilon(d_0^{cd})=0^{\,c\oplus a,\,d\oplus b}$.

Because $u\ne v$, we know $(c\oplus a, d\oplus b)\neq(0,0)$.
If $(c\oplus a, d\oplus b)=(0, 1)$, then  $uv$ is the matching edge $0^{00}0^{01}$.
Otherwise, $c\oplus a=1$, so $uv$ is one of the two possible non-edges in $B_0$. 

\textbf{Case 2}: $d_0=1$, so $u$ and $v$ lie in consecutive blocks.

Choose $(\varepsilon_0,  \varepsilon_1, \varepsilon_2)=(a,b,d)$.
Then $\rho_\varepsilon(u)=\rho_\varepsilon(0^{ab})=0^{a\oplus a, b\oplus b}=0^{00}$ and $\rho_\varepsilon(v)=\rho_\varepsilon(d_0^{cd})=1^{\,c\oplus b,\,d\oplus d}=1^{c\oplus b, 0}$.
If $b=c$, then $uv$ is the cross edge $0^{00}1^{00}$.
Otherwise, $uv$ is a non-edge $0^{00}1^{10}$ between $B_0$ and $B_1$. 

\textbf{Case 3}: $d_0\in\{2, \ldots,  q\}$. 

The two bits of $u$ and $v$ are controlled by $\varepsilon_0,\varepsilon_1$, and $\varepsilon_{d_0},\varepsilon_{d_0+1}$, respectively. 
Because $q\ge3$, we may choose the four entries independently so that both bits become $00$.  
Then $uv$ becomes $0^{00}(d_0)^{00}$ and since the blocks are neither equal nor consecutive, this pair is a non-edge.
\end{proof}

\section{The graph $G_q$ has no induced $C_{2q+2}$}
\label{sec:nocycle}

Recall that the $i$th \emph{connection} is the set of cross edges between $B_i$ and $B_{i+1}$.  
Let $Q$ be an induced cycle in $G_q$.  
Define
\begin{itemize}
\item $b_i=|V(Q)\cap B_i|$, the number of vertices of $Q$ in block $B_i$;
\item $m_i$: the number of matching edges of $Q$ within $B_i$;
\item $c_i$: the number of cross edges of $Q$ in the $i$th connection.
\end{itemize}

Because every edge of $Q$ is either a matching edge or a cross edge, we have 
\begin{equation}
  |E(Q)|=\sum_i c_i+\sum_i m_i
  \label{eq:lengthcount}
\end{equation}

By adding the vertex degrees of vertices in $V(Q)\cap B_i$, we obtain 
\begin{equation}
  2b_i=c_{i-1}+c_i+2m_i
  \label{eq:degreecount}
\end{equation}
Reducing modulo $2$ gives $c_{i-1}\equiv c_i\pmod2$.
Hence all $c_i$ have the same parity.  

Given a block $B_i$ and a bit $b\in\{0,1\}$, let
$L_i^b=\{i^{ab}\in V(Q):a\in\{0,1\}\}$ and $R_i^b=\{(i+1)^{bc}\in V(Q):c\in\{0,1\}\}$.
Also, $\ell_i^b=|L_i^b|$ and $r_i^b=|R_i^b|$. 
Since $Q$ is induced, every vertex in $L_i^b$ is adjacent to every vertex in $R_i^b$. 
Thus
\begin{equation}
  c_i=\ell_i^0r_i^0+\ell_i^1r_i^1.
  \label{eq:cutproduct}
\end{equation}

The following lemma shows that four edges in a connection force $Q$ to be a 6-cycle. 

\begin{lemma}\label{lem:fourcut}
For every $i$, we have $c_i\le4$. 
In particular, if $c_i=4$ for some $i$, then $Q$ is an induced $6$-cycle whose vertex set is in $B_i\cup B_{i+1}$.
\end{lemma}

\begin{proof}
If $\ell_i^b=r_i^b=2$, then $R_i^b$ has two vertices $(i+1)^{b0}$ and $(i+1)^{b1}$, which are endpoints of a matching edge.  
Moreover, each vertex in $R_i^b$ is also adjacent to both  vertices in $L_i^b$, so it has at least three neighbors, which contradicts that $Q$ is a cycle.
Therefore $\ell_i^br_i^b\le2$ for each $b$, so \eqref{eq:cutproduct} implies $c_i\le4$.

Suppose $c_i=4$.
Then $\ell_i^br_i^b=2$ for each $b\in\{0,1\}$. 
Assume  $(\ell_i^b, r_i^b)=(1,2)$ for some $b\in\{0,1\}$.
Then $Q$ has a $3$-cycle. 
Since $\ell_i^{1-b}\neq0\neq r_i^{1-b}$, there is an additional edge of $Q$, which contradicts that $Q$ is a cycle.  

Otherwise, $(\ell_i^0, r_i^0, \ell_i^1, r_i^1)=(2,1,2,1)$.
Thus the two vertices of $Q$ in $B_{i+1}$ are $(i+1)^{0c}$ and $(i+1)^{1d}$ for some bits $c, d$, so $Q$ is the $6$-cycle $i^{00}i^{01}(i+1)^{1d}i^{11}i^{10}(i+1)^{0c}$.
\end{proof}

\begin{lemma}
\label{lem:internalblock}
If $c_{i-1}=c_i=2$, then $m_i=0$.
\end{lemma}
\begin{proof}
Since $c_{i-1}+c_i=4$, $Q$ cannot be a $3$-cycle. 
Suppose to the contrary that $m_i\geq 1$.
Let $i^{b0}i^{b1}$ be an edge of $Q$ for some bit $b$.
If $(i-1)^{ab}\in V(Q)$ for some bit $a$, then $Q$ must be a $3$-cycle $(i-1)^{ab}i^{b0}i^{b1}$, which is a contradiction. 
In particular, both $i^{b,0}$ and $i^{b,1}$ are incident with no edges of $Q$ in the $(i-1)$st connection. 

Similarly, if both edges incident with $(i-1)^{a,1-b}$ are in $Q$ for some bit $a$, then $Q$ must be a $3$-cycle $(i-1)^{a,1-b}i^{1-b,0}i^{1-b,1}$, which is a contradiction. 
Note that $i^{1-b,0}i^{1-b,1}$ is a matching edge of $B_i$. 
Since $c_{i-1}=2$, each of $(i-1)^{0,1-b}$ and $(i-1)^{1,1-b}$ is incident with exactly one edge of $Q$ in the $(i-1)$st connection. 
In particular, there is a vertex $i^{1-b,c}$ for some bit $c$ that is incident with no edges of $Q$ in the $i$th connection since it is already incident with two edges of $Q$ in the $(i-1)$st connection. 
This further implies $(i+1)^{c,0}$ and $(i+1)^{c,1}$ are not in $Q$. 
Furthermore, $i^{b,c}$ is incident with no edges of $Q$ in the $i$th connection. 
Hence, either $i^{b,0}$ or $i^{b,1}$ is incident with only one edge in $Q$, which contradicts that $Q$ is a cycle.
\end{proof}

\begin{lemma}\label{lem:properinterval}
If $c_j=0$, then $Q$ is either a $6$-cycle from \Cref{lem:fourcut} or an odd cycle.
\end{lemma}

\begin{proof}
Since all $c_i$ have the same parity and $c_j=0$, we know  every $c_i$ is even.  
If a positive $c_i$ equals $4$, \Cref{lem:fourcut} gives the $C_6$.  
Therefore, every positive $c_i$ equals $2$.

Since $Q$ is connected, the blocks containing vertices of  $Q$ form one interval.  
Relabel them as $B_0,B_1,\ldots,B_t$, where the $t$ connections between consecutive blocks are each used twice.  
Thus $c_{-1}=c_t=0$ and $c_0=c_1=\cdots=c_{t-1}=2$.

At $B_0$, \eqref{eq:degreecount} implies $b_0=m_0+1$.
If $m_0=2$, then $c_0\geq 4$, which is a contradiction. 
If $m_0=0$, then $B_0$ contains one vertex $0^{ab}$ of $Q$. 
Since $c_0=2$, $0^{ab}1^{b0}$ and $0^{ab}1^{b1}$ are the two edges of $Q$ in the $0$th connection. 
Moreover, $1^{b0}1^{b1}$ is a matching edge, so $Q$ is a $3$-cycle, which has odd length. 
Otherwise, $m_0=1$.

For every $i\in\{1, \ldots, t-1\}$, \Cref{lem:internalblock} implies $m_i=0$.

At $B_t$, \eqref{eq:degreecount} implies $b_t=m_t+1$.
If $m_t=2$, then $c_{t-1}\geq 4$, which is a contradiction. 
If $m_t=1$, then the vertices of $Q$ in $B_t$ form a matching pair $t^{a0}t^{a1}$ for some bit $a$.  
Since $c_{t-1}=2$, $(t-1)^{ba}$ is in $Q$ for some bit $b$. 
Now, $(t-1)^{ba}t^{a0}t^{a1}$ is a $3$-cycle, so $Q$ is a $3$-cycle, which has odd length. 

Otherwise, $m_t=0$. 
Now, by \eqref{eq:lengthcount}, $|E(Q)|=2t+1$, which is odd.
\end{proof}

\begin{proposition}\label{prop:free}
The graph $G_q$ contains no induced $C_{2q+2}$.
\end{proposition}

\begin{proof}
Suppose to the contrary that $Q$ is an induced $C_{2q+2}$. 
Since $2q+2\ge8$, $Q$ is an even cycle with length at least $8$. 
By \Cref{lem:properinterval}, $Q$ must have an edge in every connection.

Recall that all $c_i$ have the same parity. 
Note that there are $2q$ connections and
$2q+2=|E(Q)|= \sum_i c_i+\sum_i m_i$.
If the parity is even, then  $c_i\ge2$ for every $i$, so $\sum_i c_i\ge4q>2q+2$, which is a contradiction.  
Hence every $c_i$ is odd.

By \Cref{lem:fourcut}, $c_i\in\{1,3\}$ for every $i$. 
If there are two connections that contain three edges of $Q$, then $\sum_i c_i\geq 2q+4$, which is a contradiction. 
Hence, there is at most one connection containing three edges of $Q$. 

Assume $c_i=1$ for every $i$, so $\sum_i m_i=2$.
For a block $B_i$ with $m_i>0$, \eqref{eq:degreecount} says
$b_i=m_i+1$. 
If $m_i=2$ for some $i$, then $b_i=4$, which violates $b_i=m_i+1$. 
Thus, there are two blocks $B_i$ and $B_j$ with $m_i=1=m_j$.
Moreover, $B_i$ contains exactly one matching pair, which has the same neighborhood in $B_{i-1}$.  This contradicts $c_{i-1}=1$.

Assume $c_j=3$ for some $j$ and $c_i=1$ for all $i\neq j$, so $m_i=0$ for every $i$.
By~\eqref{eq:degreecount}, each of $B_j$ and $B_{j+1}$ contains exactly two vertices of $Q$. 
Since $m_{j+1}=0$, the two vertices of $Q$ in $B_{j+1}$ are $(j+1)^{0a}$ and $(j+1)^{1b}$ for some bits $a,b$. 
Each vertex of $B_j$ is adjacent to exactly one of $(j+1)^{0a}$ and $(j+1)^{1b}$, so $c_j\leq 2$, which is a contradiction. 
\end{proof}

\section{Adding a non-edge}
\label{sec:addedge}

For $\beta\in\{00,01,10,11\}$ and integers $i,j$, define
\[
[i;j]_{+}^{\beta}
=
\begin{cases}
i^\beta,(i+1)^\beta,\ldots,j^\beta,& i\le j,\\
\emptyset,& i>j,
\end{cases}
\qquad\mbox{ and }\qquad
[i;j]_{-}^{\beta}
=
\begin{cases}
i^\beta,(i-1)^\beta,\ldots,j^\beta,& i\ge j,\\
\emptyset,& i<j.
\end{cases}
\]

We will prove that for every non-edge $uv$, there is an induced path of length $2q+1$ from $u$ to $v$. 
By \Cref{cor:normalization}, without loss of generality, a non-edge $uv$ is one of the following forms: 
\[
  0^{00}0^{10},\qquad
  0^{00}0^{11},\qquad
  0^{00}1^{10},\qquad
  0^{00}d^{00}\quad(2\le d\le q).
\]

Note that for $r\le s$, the graph induced on
$[r;s]_{+}^{00}\cup[r;s]_{+}^{11}$
is the disjoint union of two induced paths along consecutive blocks.

Define the following sequence of vertices: 
\begin{align*}
\sigma_{0^{10}}
  &:
  0^{00},[2q-1;2]_{-}^{00},1^{10},0^{11},0^{10} \\
\sigma_{0^{11}}
  &:
  0^{00},[2q-1;q]_{-}^{00},q^{01},
  [q+1;2q-1]_{+}^{11},0^{11} \\
\sigma_{1^{10}}
  &:
  0^{00},1^{01},2^{10},[3;2q-2]_{+}^{00},
  (2q-1)^{01},0^{11},1^{10}
\end{align*}

\tikzset{
  blocklabel/.style={
    font=\scriptsize,
    anchor=base,
    text height=1.5ex,
    text depth=0.25ex
  }
}

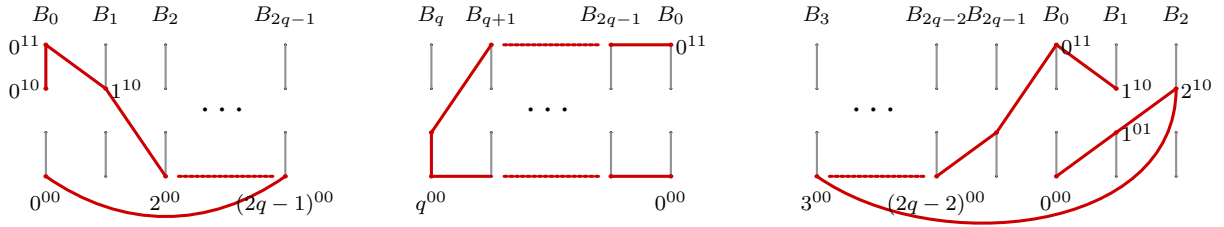
\begin{figure}[H]
\centering

\begin{tikzpicture}[
  x=0.66cm,
  y=0.58cm,
  line cap=round,
  line join=round
]

\begin{scope}[xshift=0cm]

  \foreach \name/\x in {
    0/0,
    1/1.2,
    2/2.4,
    2q-1/4.8
  }{
    \foreach \y in {0,1,2,3}{
      \fill (\x,\y) circle[radius=0.03];
    }

    \draw[matchcolor!50,line width=0.8pt]
      (\x,0)--(\x,1);

    \draw[matchcolor!50,line width=0.8pt]
      (\x,2)--(\x,3);

    \node[blocklabel] at (\x,3.55) {$B_{\name}$};
  }

  \node[font=\Large] at (3.6,1.5) {$\cdots$};

  \draw[red!80!black,line width=1.15pt]
    (0,0) to[out=-35,in=-145] (4.8,0);

  \draw[
    red!80!black,
    densely dotted,
    line width=1.15pt
  ]
    (4.55,0)--(2.65,0);

  \draw[red!80!black,line width=1.15pt]
    (2.4,0)--(1.2,2)--(0,3)--(0,2);

  \foreach \x/\y in {
    0/0,
    4.8/0,
    2.4/0,
    1.2/2,
    0/3,
    0/2
  }{
    \fill[red!80!black]
      (\x,\y) circle[radius=0.05];
  }

  \node[
    font=\scriptsize,
    anchor=north,
    yshift=-2pt
  ] at (0,0) {$0^{00}$};

  \node[
    font=\scriptsize,
    anchor=north,
    yshift=-2pt
  ] at (4.8,0) {$(2q-1)^{00}$};

  \node[
    font=\scriptsize,
    anchor=north,
    yshift=-2pt
  ] at (2.4,0) {$2^{00}$};

  \node[
    font=\scriptsize,
    anchor=west,
    xshift=-2pt
  ] at (1.2,2) {$1^{10}$};

  \node[
    font=\scriptsize,
    anchor=east,
    xshift=2pt
  ] at (0,3) {$0^{11}$};

  \node[
    font=\scriptsize,
    anchor=east,
    xshift=2pt
  ] at (0,2) {$0^{10}$};

\end{scope}

\begin{scope}[xshift=5.1cm]

  \foreach \name/\x in {
    q/0,
    q+1/1.2,
    2q-1/3.6,
    0/4.8
  }{
    \foreach \y in {0,1,2,3}{
      \fill (\x,\y) circle[radius=0.03];
    }

    \draw[matchcolor!50,line width=0.8pt]
      (\x,0)--(\x,1);

    \draw[matchcolor!50,line width=0.8pt]
      (\x,2)--(\x,3);

    \node[blocklabel] at (\x,3.55) {$B_{\name}$};
  }

  \node[font=\Large] at (2.4,1.5) {$\cdots$};

  \draw[red!80!black,line width=1.15pt]
    (4.8,0)--(3.6,0);

  \draw[
    red!80!black,
    densely dotted,
    line width=1.15pt
  ]
    (3.35,0)--(1.45,0);

  \draw[red!80!black,line width=1.15pt]
    (1.2,0)--(0,0)--(0,1)--(1.2,3);

  \draw[
    red!80!black,
    densely dotted,
    line width=1.15pt
  ]
    (1.45,3)--(3.35,3);

  \draw[red!80!black,line width=1.15pt]
    (3.6,3)--(4.8,3);

  \foreach \x/\y in {
    4.8/0,
    3.6/0,
    0/0,
    0/1,
    1.2/3,
    3.6/3,
    4.8/3
  }{
    \fill[red!80!black]
      (\x,\y) circle[radius=0.05];
  }

  \node[
    font=\scriptsize,
    anchor=north,
    yshift=-2pt
  ] at (4.8,0) {$0^{00}$};

  \node[
    font=\scriptsize,
    anchor=north,
    yshift=-2pt
  ] at (0,0) {$q^{00}$};

  \node[
    font=\scriptsize,
    anchor=west,
    xshift=-2pt
  ] at (4.8,3) {$0^{11}$};

\end{scope}

\begin{scope}[xshift=10.2cm]

  \foreach \name/\x in {
    3/0,
    2q-2/2.4,
    2q-1/3.6,
    0/4.8,
    1/6.0,
    2/7.2
  }{
    \foreach \y in {0,1,2,3}{
      \fill (\x,\y) circle[radius=0.03];
    }

    \draw[matchcolor!50,line width=0.8pt]
      (\x,0)--(\x,1);

    \draw[matchcolor!50,line width=0.8pt]
      (\x,2)--(\x,3);

    \node[blocklabel] at (\x,3.55) {$B_{\name}$};
  }

  \node[font=\Large] at (1.2,1.5) {$\cdots$};

  \draw[red!80!black,line width=1.15pt]
    (4.8,0)--(6.0,1)--(7.2,2);

  \draw[red!80!black,line width=1.15pt]
    (7.2,2) to[out=-90,in=-35] (0,0);

  \draw[
    red!80!black,
    densely dotted,
    line width=1.15pt
  ]
    (0.25,0)--(2.15,0);

  \draw[red!80!black,line width=1.15pt]
    (2.4,0)--(3.6,1)--(4.8,3)--(6.0,2);

  \foreach \x/\y in {
    4.8/0,
    6.0/1,
    7.2/2,
    0/0,
    2.4/0,
    3.6/1,
    4.8/3,
    6.0/2
  }{
    \fill[red!80!black]
      (\x,\y) circle[radius=0.05];
  }

  \node[
    font=\scriptsize,
    anchor=north,
    yshift=-2pt
  ] at (4.8,0) {$0^{00}$};

  \node[
    font=\scriptsize,
    anchor=west,
    xshift=-2pt
  ] at (6.0,1) {$1^{01}$};

  \node[
    font=\scriptsize,
    anchor=west,
    xshift=-2pt
  ] at (7.2,2) {$2^{10}$};

  \node[
    font=\scriptsize,
    anchor=north,
    yshift=-2pt
  ] at (0,0) {$3^{00}$};

  \node[
    font=\scriptsize,
    anchor=north,
    yshift=-2pt
  ] at (2.4,0) {$(2q-2)^{00}$};

  \node[
    font=\scriptsize,
    anchor=west,
    xshift=-2pt
  ] at (4.8,3) {$0^{11}$};

  \node[
    font=\scriptsize,
    anchor=west,
    xshift=-2pt
  ] at (6.0,2) {$1^{10}$};

\end{scope}

\end{tikzpicture}

\caption{
The left, middle, and right are 
$\sigma_{0^{10}}$, $\sigma_{0^{11}}$, and $\sigma_{1^{10}}$,
respectively.
}
\label{fig:basicpaths}
\end{figure}

Fix $d\in\{2, \ldots, q\}$; the target endpoint is now $d^{00}$. 

Suppose $d$ is even, and let $r=q+\frac d2$.
Define
\begin{equation*}
\begin{split}
\sigma_d^{\mathrm{even}}&:
  0^{00},(2q-1)^{10},[2q-2;r+1]_{-}^{11},r^{01},
  [r;2q-2]_{+}^{00},(2q-1)^{01},0^{11},1^{10},[2;d]_{+}^{00}.
\end{split}
\end{equation*}

Suppose $d\ge3$ is odd, and let $r=q+\frac{d+3}{2}$.
If $r\le 2q-2$, define
\begin{equation*}
\begin{split}
\sigma_d^{\mathrm{odd}}:
  0^{00},(2q-1)^{10},[2q-2;r+1]_{-}^{11},r^{01},
  [r;2q-2]_{+}^{00},(2q-1)^{01},
  [0;3]_{+}^{11},2^{01},
  [2;d]_{+}^{00}.
\end{split}
\end{equation*}
Otherwise, $r\geq 2q-1$, so $d\geq2q-5$.
Since $d\leq q$, we get $q\leq 5$. 
The only possible cases are $(q,d)\in\{(3,3),(4,3),(5,5)\}$.
Define
\begin{align*}
\sigma_{3}&:0^{00},0^{01},1^{10},2^{01},3^{11},4^{10},3^{01},3^{00} \\  
\sigma_{4}&:0^{00},1^{01},0^{10},0^{11},1^{10},2^{01},3^{11},4^{10},3^{01},3^{00} \\
\sigma_{5}&:0^{00},1^{01},2^{11},3^{11},2^{01},2^{00},3^{00},4^{01},
  5^{11},6^{10},5^{01},5^{00}.
\end{align*}

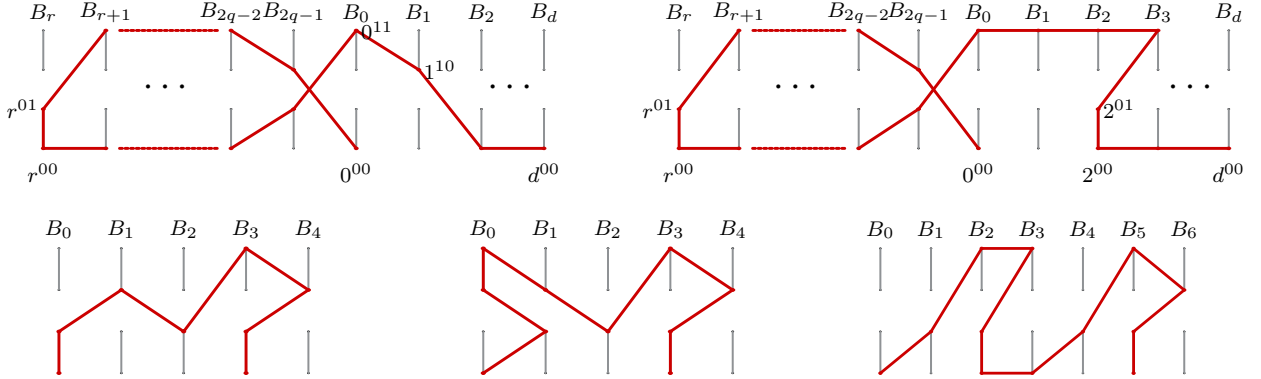
\begin{figure}[H]
\centering
\begin{minipage}{0.48\textwidth}
\centering
\begin{tikzpicture}[x=0.72cm,y=0.52cm,line cap=round,line join=round]
  \foreach \name/\x in {r/0,r+1/1.15,2q-2/3.45,2q-1/4.6,0/5.75,1/6.9,2/8.05,d/9.2}{
    \foreach \y in {0,1,2,3}{\fill (\x,\y) circle[radius=0.025];}
    \draw[matchcolor!50,line width=0.75pt] (\x,0)--(\x,1);
    \draw[matchcolor!50,line width=0.75pt] (\x,2)--(\x,3);
    \node[font=\scriptsize] at (\x,3.45) {$B_{\name}$};
  }
  \node[font=\Large] at (2.3,1.55) {$\cdots$};
  \node[font=\Large] at (8.63,1.55) {$\cdots$};
  \draw[red!80!black,line width=1.1pt] (5.75,0)--(4.6,2)--(3.45,3);
  \draw[red!80!black,densely dotted,line width=1.1pt] (3.20,3)--(1.40,3);
  \draw[red!80!black,line width=1.1pt] (1.15,3)--(0,1)--(0,0)--(1.15,0);
  \draw[red!80!black,densely dotted,line width=1.1pt] (1.40,0)--(3.20,0);
  \draw[red!80!black,line width=1.1pt] (3.45,0)--(4.6,1)--(5.75,3)--(6.9,2)--(8.05,0)--(9.2,0);
  \foreach \x/\y in {5.75/0,4.6/2,3.45/3,1.15/3,0/1,0/0,1.15/0,3.45/0,4.6/1,5.75/3,6.9/2,8.05/0,9.2/0}{\fill[red!80!black] (\x,\y) circle[radius=0.045];}
  \node[font=\scriptsize,anchor=north,yshift=-2pt] at (5.75,0) {$0^{00}$};
  \node[font=\scriptsize,anchor=east,xshift=2pt] at (0,1) {$r^{01}$};
  \node[font=\scriptsize,anchor=north,yshift=-2pt] at (0,0) {$r^{00}$};
  \node[font=\scriptsize,anchor=west,xshift=-2pt] at (5.75,3) {$0^{11}$};
  \node[font=\scriptsize,anchor=west,xshift=-2pt] at (6.9,2) {$1^{10}$};
  \node[font=\scriptsize,anchor=north,yshift=-2pt] at (9.2,0) {$d^{00}$};
\end{tikzpicture}
\end{minipage}\hfill
\begin{minipage}{0.48\textwidth}
\centering
\begin{tikzpicture}[x=0.72cm,y=0.52cm,line cap=round,line join=round]
  \foreach \name/\x in {r/0,r+1/1.1,2q-2/3.3,2q-1/4.4,0/5.5,1/6.6,2/7.7,3/8.8,d/10.1}{
    \foreach \y in {0,1,2,3}{\fill (\x,\y) circle[radius=0.025];}
    \draw[matchcolor!50,line width=0.75pt] (\x,0)--(\x,1);
    \draw[matchcolor!50,line width=0.75pt] (\x,2)--(\x,3);
    \node[font=\scriptsize] at (\x,3.45) {$B_{\name}$};
  }
  \node[font=\Large] at (2.2,1.55) {$\cdots$};
  \node[font=\Large] at (9.45,1.55) {$\cdots$};
  \draw[red!80!black,line width=1.1pt] (5.5,0)--(4.4,2)--(3.3,3);
  \draw[red!80!black,densely dotted,line width=1.1pt] (3.05,3)--(1.35,3);
  \draw[red!80!black,line width=1.1pt] (1.1,3)--(0,1)--(0,0)--(1.1,0);
  \draw[red!80!black,densely dotted,line width=1.1pt] (1.35,0)--(3.05,0);
  \draw[red!80!black,line width=1.1pt] (3.3,0)--(4.4,1)--(5.5,3)--(6.6,3)--(7.7,3)--(8.8,3)--(7.7,1)--(7.7,0)--(8.8,0)--(10.1,0);
  \foreach \x/\y in {5.5/0,4.4/2,3.3/3,1.1/3,0/1,0/0,1.1/0,3.3/0,4.4/1,5.5/3,6.6/3,7.7/3,8.8/3,7.7/1,7.7/0,8.8/0,10.1/0}{\fill[red!80!black] (\x,\y) circle[radius=0.045];}
  \node[font=\scriptsize,anchor=north,yshift=-2pt] at (5.5,0) {$0^{00}$};
  \node[font=\scriptsize,anchor=east,xshift=2pt] at (0,1) {$r^{01}$};
  \node[font=\scriptsize,anchor=north,yshift=-2pt] at (0,0) {$r^{00}$};
  \node[font=\scriptsize,anchor=west,xshift=-2pt] at (7.7,1) {$2^{01}$};
  \node[font=\scriptsize,anchor=north,yshift=-2pt] at (7.7,0) {$2^{00}$};
  \node[font=\scriptsize,anchor=north,yshift=-2pt] at (10.1,0) {$d^{00}$};
\end{tikzpicture}
\end{minipage}

\vspace{0.6em}

\begin{minipage}{0.32\textwidth}
\centering
\begin{tikzpicture}[x=0.75cm,y=0.55cm,line cap=round,line join=round]
  \foreach \name/\x in {0/0,1/1.1,2/2.2,3/3.3,4/4.4}{
    \foreach \y in {0,1,2,3}{\fill (\x,\y) circle[radius=0.025];}
    \draw[matchcolor!50,line width=0.75pt] (\x,0)--(\x,1);
    \draw[matchcolor!50,line width=0.75pt] (\x,2)--(\x,3);
    \node[font=\scriptsize] at (\x,3.45) {$B_{\name}$};
  }
  \draw[red!80!black,line width=1.1pt] (0,0)--(0,1)--(1.1,2)--(2.2,1)--(3.3,3)--(4.4,2)--(3.3,1)--(3.3,0);
  \foreach \x/\y in {0/0,0/1,1.1/2,2.2/1,3.3/3,4.4/2,3.3/1,3.3/0}{\fill[red!80!black] (\x,\y) circle[radius=0.045];}
\end{tikzpicture}
\end{minipage}\hfill
\begin{minipage}{0.32\textwidth}
\centering
\begin{tikzpicture}[x=0.75cm,y=0.55cm,line cap=round,line join=round]
  \foreach \name/\x in {0/0,1/1.1,2/2.2,3/3.3,4/4.4}{
    \foreach \y in {0,1,2,3}{\fill (\x,\y) circle[radius=0.025];}
    \draw[matchcolor!50,line width=0.75pt] (\x,0)--(\x,1);
    \draw[matchcolor!50,line width=0.75pt] (\x,2)--(\x,3);
    \node[font=\scriptsize] at (\x,3.45) {$B_{\name}$};
  }
  \draw[red!80!black,line width=1.1pt] (0,0)--(1.1,1)--(0,2)--(0,3)--(1.1,2)--(2.2,1)--(3.3,3)--(4.4,2)--(3.3,1)--(3.3,0);
  \foreach \x/\y in {0/0,1.1/1,0/2,0/3,1.1/2,2.2/1,3.3/3,4.4/2,3.3/1,3.3/0}{\fill[red!80!black] (\x,\y) circle[radius=0.045];}
\end{tikzpicture}
\end{minipage}\hfill
\begin{minipage}{0.32\textwidth}
\centering
\begin{tikzpicture}[x=0.67cm,y=0.55cm,line cap=round,line join=round]
  \foreach \name/\x in {0/0,1/1.0,2/2.0,3/3.0,4/4.0,5/5.0,6/6.0}{
    \foreach \y in {0,1,2,3}{\fill (\x,\y) circle[radius=0.025];}
    \draw[matchcolor!50,line width=0.75pt] (\x,0)--(\x,1);
    \draw[matchcolor!50,line width=0.75pt] (\x,2)--(\x,3);
    \node[font=\scriptsize] at (\x,3.45) {$B_{\name}$};
  }
  \draw[red!80!black,line width=1.05pt] (0,0)--(1.0,1)--(2.0,3)--(3.0,3)--(2.0,1)--(2.0,0)--(3.0,0)--(4.0,1)--(5.0,3)--(6.0,2)--(5.0,1)--(5.0,0);
  \foreach \x/\y in {0/0,1.0/1,2.0/3,3.0/3,2.0/1,2.0/0,3.0/0,4.0/1,5.0/3,6.0/2,5.0/1,5.0/0}{\fill[red!80!black] (\x,\y) circle[radius=0.042];}
\end{tikzpicture}
\end{minipage}
\caption{The top left and top right are $\sigma^{\mathrm{even}}_d$ and $\sigma^{\mathrm{odd}}_d$, respectively. 
The lower left, lower middle, lower right are $\sigma_{3},\sigma_{4}, \sigma_{5}$, respectively. }
\label{fig:morepaths}
\end{figure}

\begin{lemma}\label{lem:routing}
Whenever defined, each sequence $\sigma_{0^{10}}, \sigma_{0^{11}}, \sigma_{1^{10}}, \sigma^{\mathrm{even}}_{d}, \sigma^{\mathrm{odd}}_{d}, \sigma_{3}, \sigma_{4}, \sigma_{5}$ is an induced path on $2q+2$ vertices. 
\end{lemma}
\begin{proof}
It is easy to check that each sequence has $2q+2$ vertices.
Since edges of $G_q$ are only within blocks and in connections, we only need to check edges and non-edges within a block and in connections. 

Consider $\sigma_{0^{10}}$.
It has only one matching edge, which is $0^{10}0^{11}$.
It also has exactly one cross edge in every connection. 

Consider $\sigma_{0^{11}}$.
It has only one matching edge, which is $q^{00}q^{01}$.  
It also has exactly two cross edges in the $i$th connection where $i\in\{q, \ldots, 2q-1\}$. 

Consider $\sigma_{1^{10}}$.
It has no matching edges. 
It has exactly one cross edge in every connection, except for the $0$th connection where it has exactly two cross edges. 

Consider  $\sigma_d^{\mathrm{even}}$.
It has exactly one matching edge $r^{00}r^{01}$. 
It also has exactly one cross edge in the $i$th connection where $i\in\{0, \ldots, d-1\}$ and exactly two cross edges in the $i$th connection where $i\in\{r, \ldots, 2q-1\}$. 

Consider  $\sigma_d^{\mathrm{odd}}$.
It has exactly two matching edges $r^{00}r^{01}$ and $2^{00}2^{01}$. 
It also has exactly one cross edge in the $i$th connection where $i\in\{0, \ldots, d-1\}\setminus\{2\}$, exactly two cross edges in the $i$th connection where $i\in\{r, \ldots, 2q-1\}$, and exactly three cross edges in the 2nd connection. 

Each of $\sigma_3, \sigma_4, \sigma_5$ has two matching edges. 
$\sigma_3$ has exactly one cross edge in the $i$th connection where $i\in\{0,1, 2\}$ and exactly two cross edges in the 3rd connection. 
$\sigma_4$ has exactly one cross edge in the $1$st and 2nd connection, exactly two cross edges in the 3rd connection, and exactly three cross edges in the 0th connection. 
$\sigma_5$ has exactly one cross edge in the $0$th, 1st, 3rd, 4th connection, exactly two cross edges in the 5th connection, and exactly three cross edges in the 2nd connection. 

For every sequence $\sigma$, consecutive vertices in $\sigma$ are adjacent.
Moreover, since edges of $G_q$ are either  within a block or between consecutive blocks, the matching edges and cross edges listed above are all the edges induced by $\sigma$.
In each case, there are exactly  $2q+1$ edges.
Hence each sequence is an induced path on $2q+2$ vertices.
\end{proof}

\begin{proposition}\label{prop:add}
Adding any non-edge to $G_q$ creates an induced $C_{2q+2}$.
\end{proposition}

\begin{proof}
Let $uv$ be a non-edge. 
By \Cref{cor:normalization} we may assume $u=0^{00}$ and $v\in\{0^{10}, 0^{11}, 1^{10}, d^{00}\}$ where $d\in\{2, \ldots, q\}$. 
By \Cref{lem:routing}, there is an induced path on $2q+2$ vertices between $u$ and $v$.  
Thus, adding $uv$ creates an induced $C_{2q+2}$.
\end{proof}

\section{Deleting an edge}
\label{sec:deleteedge}

Recall that there are two types of edges by \Cref{cor:normalization}.

\subsection*{Deleting a matching edge}

From $G_q$, delete a matching edge, which we may assume is $0^{00}0^{01}$.  
Consider the vertex sequence 
\begin{equation}
\sigma_M:
  0^{00},(2q-1)^{00},0^{01},1^{10},
  [2;q]_{+}^{00},(q-1)^{10},[q-2;2]_{-}^{11},1^{01}. 
\label{eq:deleteM}
\end{equation}
Note that $\sigma_M$ has  $4+(q-1)+1+(q-3)+1=2q+2$ vertices.
Moreover, there is exactly one vertex in each of $B_{2q-1}$ and $B_q$, and exactly two vertices in each of $B_0, \ldots, B_{q-1}$.

Consecutive vertices of $\sigma_M$ are adjacent, including the last and first vertex.
Moreover, in every connection, the vertices of $\sigma_M$ induce either zero or two cross edges; if two cross edges are induced, then they are the two edges joining consecutive vertices of $\sigma_M$.
Since $0^{00}0^{01}$ is deleted,  vertices of $\sigma_M$ induce no edges within a block.
Hence $\sigma_M$ is an induced cycle of length $2q+2$ after $0^{00}0^{01}$ is deleted.

\begin{figure}[H]
\centering
\begin{minipage}{0.48\textwidth}
\centering
\begin{tikzpicture}[x=0.86cm,y=0.58cm,line cap=round,line join=round]
  \foreach \name/\x in {2q-1/0,0/1.2,1/2.4,2/3.6,q-2/5.4,q-1/6.6,q/7.8}{
    \foreach \y in {0,1,2,3}{ \fill (\x,\y) circle[radius=0.03]; }
    \draw[matchcolor!50,line width=0.8pt] (\x,0)--(\x,1);
    \draw[matchcolor!50,line width=0.8pt] (\x,2)--(\x,3);
    \node[font=\scriptsize] at (\x,3.55) {$B_{\name}$};
  }
  \node[font=\Large] at (4.5,1.5) {$\cdots$};
  \draw[red!80!black,line width=1.15pt] (1.2,0)--(0,0)--(1.2,1)--(2.4,2)--(3.6,0);
  \draw[red!80!black,densely dotted,line width=1.15pt] (3.85,0)--(7.55,0);
  \draw[red!80!black,line width=1.15pt] (7.8,0)--(6.6,2)--(5.4,3);
  \draw[red!80!black,densely dotted,line width=1.15pt] (5.15,3)--(3.85,3);
  \draw[red!80!black,line width=1.15pt] (3.6,3)--(2.4,1)--(1.2,0);
  \draw[red!80!black,dashed,line width=1.0pt] (1.2,0)--(1.2,1);
  \foreach \x/\y in {0/0,1.2/0,1.2/1,2.4/1,2.4/2,3.6/0,3.6/3,5.4/3,6.6/2,7.8/0}{
    \fill[red!80!black] (\x,\y) circle[radius=0.05];
  }
  \node[font=\scriptsize,anchor=north,yshift=-2pt] at (0,0) {$(2q{-}1)^{00}$};
  \node[font=\scriptsize,anchor=north,yshift=-2pt] at (1.2,0) {$0^{00}$};
  \node[font=\scriptsize,anchor=west,xshift=-2pt] at (1.2,1) {$0^{01}$};
  \node[font=\scriptsize,anchor=west,xshift=-2pt] at (2.4,2) {$1^{10}$};
  \node[font=\scriptsize,anchor=west,xshift=-2pt] at (2.4,1) {$1^{01}$};
  \node[font=\scriptsize,anchor=north,yshift=-2pt] at (3.6,0) {$2^{00}$};
  \node[font=\scriptsize,anchor=west,xshift=-2pt] at (3.6,3) {$2^{11}$};
  \node[font=\scriptsize,anchor=north,yshift=-2pt] at (7.8,0) {$q^{00}$};
\end{tikzpicture}
\end{minipage}\hfill
\begin{minipage}{0.48\textwidth}
\centering
\begin{tikzpicture}[x=0.86cm,y=0.58cm,line cap=round,line join=round]
  \foreach \name/\x in {q+2/0,q+3/1.2,2q-1/3.5,0/4.8,1/6.0,2/7.2}{
      \foreach \y in {0,1,2,3}{\fill (\x,\y) circle[radius=0.03];}
      \draw[matchcolor!50,line width=0.8pt] (\x,0)--(\x,1);
      \draw[matchcolor!50,line width=0.8pt] (\x,2)--(\x,3);
      \node[font=\scriptsize] at (\x,3.55) {$B_{\name}$};
  }
  \node[font=\Large] at (2.35,1.5) {$\cdots$};
  \draw[red!80!black,line width=1.2pt] (4.8,0)--(3.5,0);
  \draw[red!80!black,line width=1.2pt] (3.5,0)--(2.75,0);
  \draw[red!80!black,densely dotted,line width=1.2pt] (2.75,0)--(1.45,0);
  \draw[red!80!black,line width=1.2pt] (1.2,0)--(0,0)--(0,1)--(1.2,3);
  \draw[red!80!black,densely dotted,line width=1.2pt] (1.45,3)--(2.75,3);
  \draw[red!80!black,line width=1.2pt] (2.75,3)--(3.5,3)--(4.8,3)--(6.0,2)--(7.2,0);
  \draw[red!80!black,line width=1.2pt] (7.2,0)--(6.0,0)--(6.0,1)--(4.8,0);
  \draw[red!80!black,dashed,line width=1.1pt] (4.8,0)--(6.0,0);
  \foreach \x/\y in {0/0,0/1,1.2/3,3.5/0,3.5/3,4.8/0,4.8/3,6.0/0,6.0/1,6.0/2,7.2/0}{
    \fill[red!80!black] (\x,\y) circle[radius=0.05];
  }
  \node[font=\scriptsize,anchor=north,yshift=-2pt] at (0,0) {$(q{+}2)^{00}$};
  \node[font=\scriptsize,anchor=north,yshift=-2pt] at (4.8,0) {$0^{00}$};
  \node[font=\scriptsize,anchor=west,xshift=-2pt] at (4.8,3) {$0^{11}$};
  \node[font=\scriptsize,anchor=west,xshift=-2pt] at (6.0,2) {$1^{10}$};
  \node[font=\scriptsize,anchor=north,yshift=-2pt] at (7.2,0) {$2^{00}$};
  \node[font=\scriptsize,anchor=north,yshift=-2pt] at (6.0,0) {$1^{00}$};
  \node[font=\scriptsize,anchor=west,xshift=-2pt] at (6.0,1) {$1^{01}$};
\end{tikzpicture}
\end{minipage}
\caption{The left and right are when deleting a matching edge and a cross edge, respectively.
The dashed edge is the deleted edge.}
\label{fig:deletecertificates}
\end{figure}
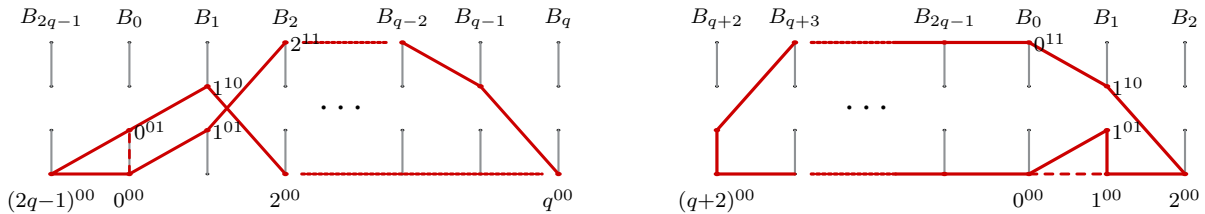

\subsection*{Deleting a cross edge}

From $G_q$, delete a cross edge, which we may assume is $0^{00}1^{00}$.  
Consider the vertex sequence 
\begin{equation}
  \sigma_C: 0^{00},[2q-1;q+2]_{-}^{00},(q+2)^{01},
  [q+3;2q-1]_{+}^{11}, 0^{11},1^{10},2^{00},1^{00},1^{01}.
  \label{eq:deleteS}
\end{equation}
Note that $\sigma_C$ contains $1+(q-2)+1+(q-3)+5=2q+2$ vertices.
Moreover, there is exactly one vertex in $B_2$, exactly three vertices in $B_1$, and exactly two vertices in each of $B_0, B_{2q-1}, \ldots, B_{q+2}$.

Consecutive vertices of $\sigma_C$  are adjacent, including the last and first vertex.
Since $0^{00}1^{00}$ is deleted, in every connection the vertices of $\sigma_C$ induce either zero or two cross edges; if two cross edges are induced, then they are the two edges joining consecutive vertices of $\sigma_C$.
Within the blocks, vertices of $\sigma_C$ induce only edges $1^{00}1^{01}$ and $(q+2)^{00}(q+2)^{01}$.
Hence $\sigma_C$ is an induced cycle of length $2q+2$ after $0^{00}1^{00}$ is deleted.

\begin{proof}[Proof of \Cref{thm:main}]
Fix an integer $q\ge3$.
By \Cref{prop:free}, the graph $G_q$ contains no induced copy of
$C_{2q+2}$.
By \Cref{prop:add}, adding any non-edge to $G_q$ creates an induced copy of
$C_{2q+2}$.

It remains to consider deleting an edge.
By \Cref{cor:normalization}, we may assume that the deleted edge is either the matching edge $0^{00}0^{01}$ or the cross edge $0^{00}1^{00}$.
The constructions in \eqref{eq:deleteM} and \eqref{eq:deleteS} show that deleting either of these edges creates an induced copy of $C_{2q+2}$.

Hence $G_q$ is $C_{2q+2}$-induced-saturated.
\end{proof}

\section*{Acknowledgments}
The author acknowledges the use of ChatGPT as an exploratory tool for
discussions on improving the initial constructions, checking case analyses,
and improving the exposition, including assistance with  first drafts of the figures.
All mathematical arguments were independently verified by the author,
who takes full responsibility for their correctness.


\end{document}